\documentclass[12pt,reqno]{amsart}

\usepackage{amsmath,amssymb,amsthm,mathtools}
\usepackage[margin=1in]{geometry}
\usepackage[hidelinks]{hyperref}
\usepackage{enumitem}
\usepackage{microtype}
\usepackage{etoolbox}

\newcommand{\Hh}{\mathbb H}
\newcommand{\R}{\mathbb R}
\newcommand{\Z}{\mathbb Z}
\newcommand{\Teich}{\mathcal T}
\newcommand{\E}{\mathcal E}
\newcommand{\Hopf}{\operatorname{Hopf}}
\newcommand{\Isom}{\operatorname{Isom}}
\newcommand{\Area}{\operatorname{Area}}

\newcommand{\Dist}{\operatorname{Dist}}

\theoremstyle{plain}
\newtheorem{theorem}{Theorem}[section]
\newtheorem{proposition}[theorem]{Proposition}
\newtheorem{lemma}[theorem]{Lemma}
\newtheorem{corollary}[theorem]{Corollary}

\theoremstyle{definition}

\newtheorem{remark}[theorem]{Remark}
\newtheorem{question}[theorem]{Question}

\title[Cannon--Thurston maps, harmonic maps, and minimal surfaces]
{Cannon--Thurston Maps, Harmonic Maps, and Minimal Surfaces\\
in Fibered Hyperbolic $3$--Manifolds}

\author{Alberto Verjovsky}
\date{September 24, 2026}

\makeatletter
\patchcmd{\@settitle}{\uppercasenonmath\@title}{}{}{}
\makeatother
\begin{document}

\begin{abstract}
Let $M_f$ be the closed hyperbolic $3$--manifold which is the mapping torus of a pseudo-Anosov homeomorphism $f$ of a closed surface $S$ of genus at least two.  The inclusion of the fiber subgroup
\(\pi_1(S)\hookrightarrow \pi_1(M_f)\) has a Cannon--Thurston boundary map
\(\partial\pi_1(S)\cong S^1\to S^2\cong\partial\pi_1(M_f)\), which is a Peano curve.  We relate this map to two variational representatives of the fiber homotopy class: the harmonic map of Eells--Sampson and Hartman, and a least-area embedded minimal surface.  Equivariant lifts of homotopic maps are a bounded distance apart; hence, for every marked conformal structure on $S$, the lift of the harmonic representative extends continuously to the closed disk with boundary value the Cannon--Thurston map, and a least-area fiber lifts to a properly embedded minimal plane in $\Hh^3$ with the same boundary map.  We give area and curvature bounds for a least-area fiber, show that its largest principal curvature is at least one, and study the energy of the harmonic representatives as a function on Teichm\"uller space, its critical points, which are the zeros of the Hopf differential, and its invariance under the monodromy.
\end{abstract}
\maketitle

\section{The fiber subgroup and its boundary map}

Let $S$ be a closed connected orientable surface of genus $g\geq2$, and let
$f:S\to S$ be an orientation-preserving pseudo-Anosov homeomorphism: there are transverse measured singular foliations
\((\mathcal F^s,\mu^s)\), \((\mathcal F^u,\mu^u)\) and a number $\lambda>1$ such that
\[
f(\mathcal F^s,\mu^s)=(\mathcal F^s,\lambda^{-1}\mu^s),
\qquad
f(\mathcal F^u,\mu^u)=(\mathcal F^u,\lambda\mu^u).
\]
For a choice of hyperbolic metric on $S$ the two foliations correspond to measured geodesic laminations, the stable and unstable laminations of $f$ \cite{FLP}.

The mapping torus
\[
M_f=S\times[0,1]/(x,1)\sim(f(x),0)
\]
is a closed hyperbolic $3$--manifold by Thurston's hyperbolization theorem for fibered $3$--manifolds \cite{Th,Ot,McM}; we fix its hyperbolic metric, and identify $\partial\Hh^3=S^2$.  Let $j:S\to M_f$, $j(x)=[x,0]$, be the fiber inclusion, $G=\pi_1(S)$ and $\Gamma=\pi_1(M_f)$.  There is an exact sequence
\[
1\longrightarrow G\xrightarrow{\;\iota\;} \Gamma\longrightarrow\Z\longrightarrow1,
\]
and we regard $G$ as a normal subgroup of the cocompact Kleinian group $\Gamma<\Isom^+(\Hh^3)$.  We fix $t\in\Gamma$ mapping to a generator of $\Z$; conjugation by $t$ preserves $G$, and its restriction to $G$ represents the outer automorphism class of $f_*$ (for the appropriate choice of generator).

A \emph{marked hyperbolic structure} $X$ on $S$ identifies the universal cover $\widetilde S$ with $\Hh^2$ and gives a discrete faithful representation $\rho_X:G\to\Isom^+(\Hh^2)$ by deck transformations.  Fix a lift $\widetilde j:\Hh^2\to\Hh^3$ of $j$.  There is a unique injective homomorphism $\rho:G\to\Gamma<\Isom^+(\Hh^3)$ with image $G$ such that
\[
\widetilde j\circ\rho_X(\gamma)=\rho(\gamma)\circ\widetilde j\qquad(\gamma\in G).
\]
A map $u:\Hh^2\to\Hh^3$ satisfying $u\circ\rho_X(\gamma)=\rho(\gamma)\circ u$ for all $\gamma\in G$ is called \emph{$\rho$--equivariant}.  The orbit map $\gamma\mapsto\rho_X(\gamma)z_0$ is a quasi-isometry $G\to\Hh^2$ and induces a homeomorphism $\partial G\cong S^1=\partial\Hh^2$ which does not depend on $z_0$.  For two marked hyperbolic structures $X,Y$ the resulting homeomorphism $\partial\Hh^2_X\to\partial\Hh^2_Y$ is the boundary extension of any $G$--equivariant quasi-isometry $\Hh^2_X\to\Hh^2_Y$, for instance the lift of the identity of $S$.  In what follows, boundary maps defined on $\partial\Hh^2$ are regarded as maps on $\partial G$ through this identification.

A \emph{Peano curve} is a continuous surjection $S^1\to S^2$.

\begin{theorem}[Cannon--Thurston \cite{CT}; see also \cite{Mi1,Mi2,Mj}]
\label{thm:CT}
The lift $\widetilde j:\Hh^2\to\Hh^3$ extends continuously to a map
\[
\widehat j:\Hh^2\cup S^1\longrightarrow\Hh^3\cup S^2 .
\]
The boundary map $\partial\iota:=\widehat j|_{S^1}:S^1\to S^2$ is $\rho$--equivariant and surjective; in particular it is a Peano curve.
\end{theorem}

Mitra \cite{Mi2} proved the corresponding statement for the inclusion of a hyperbolic normal subgroup in a hyperbolic group: $\iota$ extends continuously to $G\cup\partial G\to\Gamma\cup\partial\Gamma$.  Surjectivity holds because the image of $\partial\iota$ is the limit set of $G$, and the limit set of a nontrivial normal subgroup of the cocompact group $\Gamma$ is $\Lambda(\Gamma)=S^2$.  Since a continuous injection $S^1\to S^2$ is not onto, $\partial\iota$ is not injective.

\section{Bounded distance and boundary maps}

\begin{lemma}
\label{lem:bounded}
Let $u,v:\Hh^2\to\Hh^3$ be continuous maps with
\[
\sup_{z\in\Hh^2}d_{\Hh^3}(u(z),v(z))<\infty.
\]
Suppose that $v$ extends continuously to
\(\widehat v:\Hh^2\cup S^1\to \Hh^3\cup S^2\) with $\widehat v(S^1)\subset S^2$.  Then $u$ has a unique continuous extension $\widehat u:\Hh^2\cup S^1\to \Hh^3\cup S^2$, and
\(\widehat u|_{S^1}=\widehat v|_{S^1}\).
\end{lemma}

\begin{proof}
Let $C=\sup_z d(u(z),v(z))$ and define $\widehat u=u$ on $\Hh^2$ and $\widehat u=\widehat v$ on $S^1$.  Both compactifications are metrizable, so it suffices to check sequential continuity at points $\xi\in S^1$.  Let $z_n\to\xi$ in $\Hh^2\cup S^1$.  For the terms $z_n\in S^1$ we have $\widehat u(z_n)=\widehat v(z_n)\to\widehat v(\xi)$.  For the terms $z_n\in\Hh^2$ we have $v(z_n)\to\widehat v(\xi)\in S^2$ and $d(u(z_n),v(z_n))\le C$.  In $\Hh^3$, if $x_n\to\zeta\in S^2$ and $d(x_n,y_n)\le C$, then the Gromov products satisfy $(x_n\cdot y_n)_o\ge d(o,x_n)-C\to\infty$, so $y_n\to\zeta$.  Hence $u(z_n)\to\widehat v(\xi)$.  Uniqueness holds because $\Hh^2$ is dense in $\Hh^2\cup S^1$ and the target is Hausdorff.
\end{proof}

\begin{lemma}
\label{lem:equivariantbounded}
Let a group $G$ act on a topological space $Y$ and by isometries on a metric space $Z$, and suppose there is a compact set $K\subset Y$ with $G\cdot K=Y$.  If $a,b:Y\to Z$ are continuous and $G$--equivariant, then
\(\sup_{y\in Y}d_Z(a(y),b(y))<\infty\).
\end{lemma}

\begin{proof}
The function $y\mapsto d(a(y),b(y))$ is continuous, hence bounded on $K$, and it is $G$--invariant because $d(a(\gamma y),b(\gamma y))=d(\gamma a(y),\gamma b(y))=d(a(y),b(y))$.
\end{proof}

\begin{proposition}
\label{prop:anyrepresentative}
Let $X$ be a marked hyperbolic structure on $S$ and let $u:\Hh^2\to\Hh^3$ be a continuous $\rho$--equivariant map.  Then $u$ extends continuously to $\Hh^2\cup S^1$, and its boundary map is $\partial\iota$.  In particular this holds for the $\rho$--equivariant lift of any continuous map $a:S\to M_f$ homotopic to $j$.
\end{proposition}

\begin{proof}
Apply Lemma~\ref{lem:equivariantbounded} to $u$ and $\widetilde j$, with $K$ a closed fundamental polygon for $\rho_X(G)$; then apply Lemma~\ref{lem:bounded} with $v=\widetilde j$ and Theorem~\ref{thm:CT}.  If $H:S\times[0,1]\to M_f$ is a homotopy from $j$ to $a$, lifting $H$ with initial map $\widetilde j$ gives a lift $\widetilde a$ of $a$; for each $\gamma\in G$ the maps $\widetilde H_s\circ\rho_X(\gamma)$ and $\rho(\gamma)\circ\widetilde H_s$ are lifts of $H_s$ agreeing at $s=0$, hence for all $s$, so $\widetilde a$ is $\rho$--equivariant.
\end{proof}

\begin{remark}
\label{rem:otherlifts}
Every other lift of $a$ has the form $\delta\circ\widetilde a$ with $\delta\in\Gamma$; it is equivariant for $\gamma\mapsto\delta\rho(\gamma)\delta^{-1}$, and its boundary map is $\delta\circ\partial\iota$.  Such a map has the same image $S^2$ as $\partial\iota$.  If $\delta\in G$, it induces the same equivalence relation on $S^1$ as $\partial\iota$: indeed $\delta\circ\partial\iota=\partial\iota\circ\rho_X(\delta)$, so postcomposition by $\delta$ only relabels the image points.
\end{remark}

We also need the maps $S\to M_f$ that are homotopic to $j$ up to a change of marking.

\begin{lemma}
\label{lem:markings}
Let $\psi:S\to S$ be a homeomorphism.  Then $j\circ\psi$ is freely homotopic to $j$ if and only if $\psi$ is isotopic to $f^n$ for some $n\in\Z$.
\end{lemma}

\begin{proof}
The homotopy $H_s(x)=[x,s]$, $s\in[0,1]$, goes from $j$ to $j\circ f$; iterating, $j\circ f^n\simeq j$ for all $n$.  Conversely, $M_f$ is aspherical, so free homotopy classes of maps $S\to M_f$ correspond to $\Gamma$--conjugacy classes of homomorphisms $G\to\Gamma$.  If $j\circ\psi\simeq j$, then $\iota\circ\psi_*=c_\delta\circ\iota$ for some $\delta\in\Gamma$, where $\psi_*$ is defined up to an inner automorphism of $G$.  Writing $\delta=\gamma t^n$ with $\gamma\in G$, the outer class of $\psi_*$ equals that of $f_*^n$.  By the Dehn--Nielsen--Baer theorem and the fact that homotopic homeomorphisms of a closed surface are isotopic \cite{FM}, $\psi$ is isotopic to $f^n$.
\end{proof}

\section{Harmonic representatives}

Let $\Teich=\Teich(S)$ be the Teichm\"uller space of marked conformal structures on $S$.  Each $X\in\Teich$ is represented by its hyperbolic metric $\sigma_X$, so the preceding section applies.  For a mapping class represented by a diffeomorphism $\varphi$ we write $\varphi^*X$ for the pulled-back marked structure; it depends only on the isotopy class of $\varphi$.  Since $f$ need not be smooth at the singularities of its foliations, we fix a diffeomorphism $\phi$ isotopic to $f$ and set $f^*X:=\phi^*X$.

\begin{theorem}[Eells--Sampson \cite{ES}, Hartman \cite{Hartman}]
\label{thm:ESHartman}
For every $X\in\Teich$ there is a unique harmonic map $h_X:(S,\sigma_X)\to M_f$ freely homotopic to $j$.
\end{theorem}

\begin{proof}
Existence follows from the Eells--Sampson theorem, since $S$ and $M_f$ are compact and $M_f$ has sectional curvature $-1$.  By Hartman's theorem, two homotopic harmonic maps into a compact manifold of negative sectional curvature coincide unless their image is contained in a closed geodesic or is a point.  A map with image in a closed geodesic induces a homomorphism with cyclic image, whereas maps homotopic to $j$ induce a homomorphism with image conjugate to the nonabelian group $G$.
\end{proof}

Harmonicity of $h_X$ depends only on the conformal class of $\sigma_X$, since the energy of a map from a surface is conformally invariant.

\begin{theorem}
\label{thm:harmonicCT}
For every $X\in\Teich$, the $\rho$--equivariant lift
\(\widetilde h_X:\Hh^2\to\Hh^3\)
of $h_X$ extends continuously to
\(\widehat h_X:\Hh^2\cup S^1\to\Hh^3\cup S^2\), and
\(\widehat h_X|_{S^1}=\partial\iota\).
In particular
\(\widehat h_X(S^1)=S^2\).
\end{theorem}

\begin{proof}
Apply Proposition~\ref{prop:anyrepresentative} to $a=h_X$.
\end{proof}

\begin{corollary}
\label{cor:Xindependent}
Let $X,Y\in\Teich$.  Under the identifications $\partial\Hh^2_X\cong\partial G\cong\partial\Hh^2_Y$ of Section~1, the boundary maps of $\widetilde h_X$ and $\widetilde h_Y$ coincide.
\end{corollary}

\subsection{The Hopf differential}

For a smooth map $h:X\to M_f$ from a Riemann surface, the Hopf differential is the $(2,0)$--part of the pullback metric,
\[
\Hopf(h)=(h^*g_{M_f})^{2,0}=\langle h_z,h_z\rangle\,dz^2
\]
in a local complex coordinate $z=x+iy$, where $\langle\cdot,\cdot\rangle$ is extended complex-bilinearly.  If $h$ is harmonic, then $\Hopf(h)$ is holomorphic, and $\Hopf(h)\equiv0$ if and only if $h$ is weakly conformal.  A nonconstant weakly conformal harmonic map from a surface is a branched minimal immersion \cite{GOR}.  For $X\in\Teich$ we write
\[
\Phi_X:=\Hopf(h_X)\in Q(X),
\]
where $Q(X)$ is the space of holomorphic quadratic differentials on $X$.  Since $h_X$ is nonconstant, $h_X$ is a branched minimal immersion if and only if $\Phi_X=0$.

\section{Least-area fibers}

Let
\[
A_0=\inf\{\Area(u):u:S\to M_f \text{ smooth},\ u\simeq j\},
\]
where $\Area(u)$ is the parametrized area, counted with multiplicity.

\begin{theorem}
\label{thm:leastarea}
There is a smooth embedded surface $\Sigma_{\min}\subset M_f$, isotopic to the fiber $j(S)$, with $\Area(\Sigma_{\min})=A_0$.  In particular $\Sigma_{\min}$ is a stable minimal surface.
\end{theorem}

\begin{proof}
Since $j_*$ is injective, the least-area existence theorem for incompressible maps gives a smooth least-area immersion in this homotopy class; see Schoen--Yau \cite{SY}, Sacks--Uhlenbeck \cite{SU}, and the formulation in Hass--Scott \cite{HS}.  Freedman--Hass--Scott \cite{FHS} show that, because the class contains the two-sided embedding $j$, a least-area immersion is either an embedding or a double cover of an embedded one-sided surface $P$.  The manifold $M_f$ is orientable because $S$ is orientable and $f$ is orientation preserving.  Hence a one-sided surface $P$ has a nontrivial orientation double cover.  The second case cannot occur.  Indeed, the image of $\pi_1(P)$ in $\Gamma$ would contain $G$ with index two.  Since $G\triangleleft\Gamma$, its image in $\Gamma/G\cong\Z$ would therefore be finite, hence trivial; thus $\pi_1(P)$ would be contained in $G$, contradicting $[\pi_1(P):G]=2$.  Hence the least-area immersion is an embedding.  It is incompressible and homotopic to the fiber; since $M_f$ is Haken, Waldhausen's theorem \cite{Wald} implies that it is isotopic to the fiber.  A global area minimizer in its homotopy class is minimal, and its second variation is nonnegative, so it is stable.
\end{proof}

We call such a surface a \emph{least-area fiber}.  Choose a diffeomorphism $u_0:S\to\Sigma_{\min}$ with $u_0\simeq j$, and let $X_{\min}\in\Teich$ be the marked conformal structure of $u_0^*g_{M_f}$.  By Lemma~\ref{lem:markings}, another such choice changes $X_{\min}$ to $(f^n)^*X_{\min}$ for some $n\in\Z$.

\begin{theorem}
\label{thm:minimalPeano}
The $\rho$--equivariant lift $F:\Hh^2\to\Hh^3$ of $u_0$, with $\Hh^2$ the universal cover of $X_{\min}$, has the following properties:
\begin{enumerate}[label=\textup{(\roman*)}]
\item $F$ is a conformal harmonic embedding, and $F(\Hh^2)$ is a connected component of the preimage of $\Sigma_{\min}$ in $\Hh^3$;
\item $F(\Hh^2)$ is complete in the induced metric and properly embedded in $\Hh^3$;
\item $F$ extends continuously to
\(\widehat F:\Hh^2\cup S^1\to\Hh^3\cup S^2\);
\item $\widehat F|_{S^1}=\partial\iota$; in particular $\widehat F(S^1)=S^2$.
\end{enumerate}
\end{theorem}

\begin{proof}
Let $\widetilde\Sigma$ be the component of the preimage of $\Sigma_{\min}$ containing $F(\Hh^2)$.  The covering $\widetilde\Sigma\to\Sigma_{\min}$ corresponds to the kernel of $\pi_1(\Sigma_{\min})\to\pi_1(M_f)$, which is trivial; so $\widetilde\Sigma$ is simply connected, and $F:\Hh^2\to\widetilde\Sigma$ is a lift of the diffeomorphism $u_0$ between universal covers, hence a diffeomorphism.  Since $u_0$ is conformal and minimal, $F$ is conformal and harmonic.  As $F$ is a diffeomorphism onto the embedded component $\widetilde\Sigma$, it is an embedding.  The preimage of the closed set $\Sigma_{\min}$ is a closed embedded surface in $\Hh^3$, and its components are closed; so $\widetilde\Sigma$ is properly embedded.  It is complete in its induced metric because it covers the compact surface $\Sigma_{\min}$.  Items (iii) and (iv) follow from Proposition~\ref{prop:anyrepresentative} and Theorem~\ref{thm:CT}.
\end{proof}

By Remark~\ref{rem:otherlifts}, the other components of the preimage of $\Sigma_{\min}$ are the translates $t^nF(\Hh^2)$, with boundary maps $t^n\circ\partial\iota$.

In the asymptotic Plateau problem one prescribes a Jordan curve in $S^2$ and seeks a complete minimal disk in $\Hh^3$ with that curve as its asymptotic boundary \cite{Anderson}.  In Theorem~\ref{thm:minimalPeano} the boundary map of the minimal plane is instead a Peano curve onto all of $S^2$.

\begin{corollary}
\label{cor:Hopfzero}
With $X_{\min}$ as above, $h_{X_{\min}}=u_0$.  Consequently
\(\Phi_{X_{\min}}=0\).
\end{corollary}

\begin{proof}
The map $u_0:(S,X_{\min})\to M_f$ is conformal and minimal, hence harmonic, and it is homotopic to $j$.  Theorem~\ref{thm:ESHartman} gives $u_0=h_{X_{\min}}$, and $\Hopf(u_0)=0$ because $u_0$ is conformal.
\end{proof}

The uniqueness of the least-area fiber is not asserted.  Hass \cite{Hass} showed that some hyperbolic $3$--manifolds fibering over the circle admit no fibration by minimal surfaces.

\section{Stability, curvature, distortion, energy, and monodromy}

\subsection{Area and curvature of a stable minimal fiber}

Let $\Sigma\subset M_f$ be a closed embedded minimal surface isotopic to the fiber, with second fundamental form $A$, principal curvatures $\pm k$ ($k\geq0$), and Gauss curvature $K_\Sigma$.  By the Gauss equation,
\[
|A|^2=2k^2,
\qquad
K_\Sigma=-1-k^2=-1-\tfrac12|A|^2 .
\]

\begin{proposition}
\label{prop:area-curvature}
If $\Sigma$ is stable, then
\[
2\pi(g-1)< \Area(\Sigma)<4\pi(g-1),
\qquad
0<\int_{\Sigma}|A|^2\,dA=8\pi(g-1)-2\Area(\Sigma)<4\pi(g-1),
\]
and
\[
-2<\frac{1}{\Area(\Sigma)}\int_{\Sigma}K_{\Sigma}\,dA<-1 .
\]
The upper bound $\Area(\Sigma)<4\pi(g-1)$ holds without the stability assumption.
\end{proposition}

\begin{proof}
Integrating the Gauss equation and using Gauss--Bonnet,
\begin{equation}
\label{eq:GB}
\Area(\Sigma)+\frac12\int_{\Sigma}|A|^2\,dA=4\pi(g-1).
\end{equation}
Thus $\Area(\Sigma)\leq4\pi(g-1)$, with equality exactly when $A\equiv0$.  In that case a lift of $\Sigma$ to $\Hh^3$ is a totally geodesic plane invariant under a conjugate of $G$, and the limit set of $G$ would be a round circle.  The limit set of $G$ is $S^2$ (Section~1), so the upper bound is strict.

The Jacobi operator of $\Sigma$ is $L=\Delta+|A|^2+\operatorname{Ric}(N,N)=\Delta+|A|^2-2$.  Stability, applied to the constant test function $1$, gives $\int_\Sigma(|A|^2-2)\,dA\le0$.  Together with \eqref{eq:GB} this gives $\Area(\Sigma)\geq2\pi(g-1)$.  If equality holds, the quadratic form $Q(\varphi)=\int_\Sigma(|\nabla\varphi|^2-(|A|^2-2)\varphi^2)\,dA$, which is nonnegative, vanishes at $\varphi=1$; so $1$ lies in the kernel of $L$, and $|A|^2\equiv2$, that is, $k\equiv1$.  A surface in a hyperbolic $3$--manifold whose principal curvatures are constant and distinct is flat: in a principal frame $e_1,e_2$, the Codazzi equations give $(k_1-k_2)\,\omega_{12}(e_i)=0$ for $i=1,2$, so $\omega_{12}=0$.  The Cartan structure equation $d\omega_{12}=-K_\Sigma\,\omega_1\wedge\omega_2$ then gives $K_\Sigma=0$.  This contradicts $K_\Sigma=-2$.  Hence the lower bound is strict.  The remaining statements follow from \eqref{eq:GB} and $\int_\Sigma K_\Sigma\,dA=-4\pi(g-1)$.
\end{proof}

\subsection{The largest principal curvature}

\begin{proposition}[Uhlenbeck; see also Epstein]
\label{prop:maxcurv}
Let $u:S\to M_f$ be a smooth immersion homotopic to $j$.  Then at some point of $S$ some principal curvature of $u$ has absolute value at least $1$.  In particular every closed minimal surface in $M_f$ homotopic to the fiber satisfies
\[
\max_\Sigma k\geq1,\qquad \max_\Sigma|A|^2\geq2,\qquad \min_\Sigma K_\Sigma\leq-2 .
\]
\end{proposition}

\begin{proof}
Suppose that all principal curvatures of $u$ have absolute value at most $k_0<1$; by compactness this is the negation of the conclusion.  Let $\widetilde u:\Hh^2\to\Hh^3$ be the $\rho$--equivariant lift, with the induced complete metric $g_u$ and unit normal $N$, and let $E:\Hh^2\times\R\to\Hh^3$, $E(x,s)=\exp_{\widetilde u(x)}(sN(x))$.  Jacobi fields in curvature $-1$ give
\[
E^*g_{\Hh^3}=ds^2+g_u\big((\cosh s\,I-\sinh s\,W)\cdot,(\cosh s\,I-\sinh s\,W)\cdot\big),
\]
where $W$ is the shape operator.  The eigenvalues of $\cosh s\,I-\sinh s\,W$ are $\cosh s-k_i\sinh s=\sqrt{1-k_i^2}\,\cosh(s-\operatorname{artanh}k_i)\geq\sqrt{1-k_0^2}$.  Hence
\[
E^*g_{\Hh^3}\geq ds^2+(1-k_0^2)\,g_u .
\]
The right side is a complete metric, so $E^*g_{\Hh^3}$ is complete.  Since $E^*g_{\Hh^3}$ is positive definite, $dE$ is nonsingular everywhere; thus $E$ is a local diffeomorphism, and with the pullback metric it is a local isometry.  By the standard covering theorem for complete local isometries, $E$ is a covering map onto the connected manifold $\Hh^3$; since $\Hh^3$ is simply connected, $E$ is a diffeomorphism.  Let $P=\mathrm{pr}_1\circ E^{-1}:\Hh^3\to\Hh^2$.  By the inequality above, $P$ multiplies lengths of curves by at most $(1-k_0^2)^{-1/2}$, and $P\circ\widetilde u=\mathrm{id}$.  Therefore
\[
d_{g_u}(x,y)\leq(1-k_0^2)^{-1/2}\,d_{\Hh^3}(\widetilde u(x),\widetilde u(y)),
\]
and $d_{\Hh^3}(\widetilde u(x),\widetilde u(y))\le d_{g_u}(x,y)$.  Since $g_u$ is $G$--invariant and $S$ is compact, $g_u$ is bi-Lipschitz to $\sigma_X$ for any $X$.  So $\widetilde u$ is a quasi-isometric embedding, and its boundary extension is injective.  By Proposition~\ref{prop:anyrepresentative} this extension is $\partial\iota$, which is not injective.  For a minimal surface the principal curvatures are $\pm k$, and the remaining inequalities follow from the Gauss equation.
\end{proof}

The argument is the one used by Uhlenbeck \cite{Uh} and Epstein \cite{Epstein} for surfaces with principal curvatures in $(-1,1)$; the statement for least-area fibers is attributed to Uhlenbeck in \cite{FVP}.  Farre and Vargas Pallete \cite{FVP} obtain stronger lower bounds along certain families.  Huang and Lowe \cite{HL} prove a genus-dependent gap theorem: for every $g\ge2$ there is $\varepsilon(g)>0$ such that, in every closed hyperbolic $3$--manifold fibering with fiber genus $g$, every embedded minimal surface isotopic to the fiber has maximum principal curvature greater than $1+\varepsilon(g)$.

\begin{corollary}
\label{cor:A2range}
For a stable minimal surface $\Sigma$ isotopic to the fiber,
\[
\min_\Sigma|A|^2<2\leq\max_\Sigma|A|^2 .
\]
\end{corollary}

\begin{proof}
By Proposition~\ref{prop:area-curvature}, the mean value of $|A|^2$ is $(8\pi(g-1)-2\Area(\Sigma))/\Area(\Sigma)$, which is less than $2$ because $\Area(\Sigma)>2\pi(g-1)$.  The upper inequality is Proposition~\ref{prop:maxcurv}.
\end{proof}

\subsection{Comparison with the uniformizing metric}

\begin{theorem}
\label{thm:strict-contraction}
Let $F:\Hh^2\to\Hh^3$ be the map of Theorem~\ref{thm:minimalPeano}, and let $\sigma$ be the hyperbolic metric on $\Hh^2$, the universal cover of $X_{\min}$.  Then $F^*g_{\Hh^3}<\sigma$ at every point, and there is $c\in(0,1)$ with $F^*g_{\Hh^3}\leq c^2\sigma$.  In particular
\[
d_{\Hh^3}(F(x),F(y))\leq c\,d_{\sigma}(x,y)
\qquad(x,y\in\Hh^2).
\]
\end{theorem}

\begin{proof}
Write $F^*g_{\Hh^3}=e^{2w}\sigma$, where $w$ is $G$--invariant and descends to $S$.  The curvature of $e^{2w}\sigma$ is $K=e^{-2w}(-1-\Delta_\sigma w)$, and $K=-1-\tfrac12|A|^2\le-1$; hence
\[
\Delta_\sigma w\ \geq\ e^{2w}-1 .
\]
At a maximum point of $w$ on the compact surface $S$ we have $\Delta_\sigma w\le0$, so $w\le0$ everywhere (this is the Ahlfors--Schwarz lemma \cite{Ahlfors} in the present setting).  Write $e^{2w}-1=\beta w$ with $\beta=(e^{2w}-1)/w\geq0$ (and $\beta=2$ where $w=0$).  Then $\Delta_\sigma w-\beta w\geq0$ and $w\le0$; by the strong maximum principle either $w\equiv0$ or $w<0$ everywhere (compare \cite{Minda}).  If $w\equiv0$, then $K\equiv-1$ and $A\equiv0$, which is excluded by Proposition~\ref{prop:area-curvature}.  Hence $w<0$, and $c=e^{\max_S w}<1$.  Integrating along curves gives the distance estimate.
\end{proof}

\subsection{Exponential distortion of the minimal plane}

For $x,y$ in the plane $\widetilde\Sigma=F(\Hh^2)$ let $d_{\widetilde\Sigma}(x,y)$ be the intrinsic distance, and set
\[
\Dist(r)=\sup\{d_{\widetilde\Sigma}(x,y):x,y\in\widetilde\Sigma,\ d_{\Hh^3}(x,y)\leq r\}.
\]

\begin{theorem}
\label{thm:exponential-distortion}
There are constants $C\ge1$ and $0<a\le b$ such that
\[
C^{-1}e^{ar}-C\ \leq\ \Dist(r)\ \leq\ Ce^{br}\qquad(r\geq1).
\]
In particular $\widetilde\Sigma$ is not quasi-isometrically embedded in $\Hh^3$ and is not quasiconvex.
\end{theorem}

\begin{proof}
The group $G$ acts properly and cocompactly by isometries on $\widetilde\Sigma$ with its intrinsic metric, and $\Gamma$ acts properly and cocompactly on $\Hh^3$.  Fix $x_0\in\widetilde\Sigma$ and finite generating sets $\mathcal S_G$ of $G$ and $\mathcal S_G\cup\{t\}$ of $\Gamma$.  By the Milnor--\v{S}varc lemma, $\gamma\mapsto\gamma x_0$ is a quasi-isometry from $(G,|\cdot|_G)$ to $\widetilde\Sigma$ and from $(\Gamma,|\cdot|_\Gamma)$ to $\Hh^3$; the points of $\widetilde\Sigma$ are within bounded intrinsic distance of $Gx_0$.  It therefore suffices to prove the corresponding estimates for $\max\{|\gamma|_G:\gamma\in G,\ |\gamma|_\Gamma\le n\}$.

Let $\alpha=c_t|_G$.  For the lower bound, take $\gamma\in G\setminus\{1\}$.  The element $t^n\gamma t^{-n}=\alpha^n(\gamma)$ has $|\alpha^n(\gamma)|_\Gamma\leq2n+|\gamma|_\Gamma$.  The conjugacy class of $\alpha^n(\gamma)$ is the class of $f^{\pm n}(\gamma)$, because $\alpha$ represents the outer automorphism class of $f_*^{\pm1}$.  On the fixed compact hyperbolic surface $(S,\sigma_X)$, conjugacy length in $G$ is comparable with the length of the corresponding closed geodesic, and the latter grows exponentially like $\lambda^n$ for every nontrivial conjugacy class under a pseudo-Anosov map \cite{FLP}.  Since ordinary word length dominates conjugacy length, this gives the required exponential lower bound.  For the upper bound, a word of length $n$ in $\mathcal S_G\cup\{t^{\pm1}\}$ representing an element of $G$ has total $t$--exponent zero, so it can be rewritten as a product of at most $n$ elements $t^{e}st^{-e}=\alpha^{e}(s)$ with $s\in\mathcal S_G$ and $|e|\le n$.  Put $D=\max\bigl(\{2\}\cup\{|\alpha^{\pm1}(s)|_G:s\in\mathcal S_G\}\bigr)$.  Iterating the Lipschitz bounds for $\alpha$ and $\alpha^{-1}$ on word length gives $|\alpha^e(s)|_G\le D^{|e|}\le D^n$, so the element has $G$--length at most $nD^n$.  These estimates are first obtained at integer radii in the word metrics; the quasi-isometry constants from the Milnor--\v{S}varc lemma and the monotonicity of $\Dist(r)$ convert them, after changing the constants, into the stated inequalities for every real $r\ge1$.  See also \cite{NguyenSun} for distortion of subgroups of $3$--manifold groups.
\end{proof}

A quasi-isometric embedding $\Hh^2\to\Hh^3$ has an injective boundary extension, so the last assertion also follows from the non-injectivity of $\partial\iota$.  Theorems~\ref{thm:strict-contraction} and \ref{thm:exponential-distortion} concern the same map $F$: it contracts the hyperbolic metric of $X_{\min}$ by the factor $c<1$, while intrinsic and extrinsic distances on its image differ exponentially on large scales.

\subsection{The energy on Teichm\"uller space}

For $X\in\Teich$ set
\[
\E(X)=E(h_X)=\frac12\int_S|dh_X|^2\,dA_{\sigma_X}.
\]
The Jacobi operator of $h_X$ has trivial kernel: by the second variation formula, a Jacobi field $V$ along $h_X$ satisfies $\nabla V=0$ and $V\wedge dh_X(e)=0$ for every tangent vector $e$, so if $V\neq0$ then $dh_X$ has rank at most one everywhere and $h_X$ maps into a geodesic, which is impossible.  Hence $h_X$ depends smoothly on $X$ \cite{EL}, and $\E$ is differentiable.  Represent a tangent vector to $\Teich$ at $X$ by a Beltrami differential $\mu$.  Since $h_X$ is a critical point of the energy for fixed $X$, only the variation of the conformal structure contributes, and a computation in local coordinates gives
\[
d\E_X(\mu)=-4\operatorname{Re}\int_X\Phi_X\,\mu ,
\]
where $\Phi_X\mu$ is the $(1,1)$--form $\varphi(z)\mu(z)\,dx\,dy$ for $\Phi_X=\varphi\,dz^2$, $\mu=\mu(z)\,d\bar z/dz$.  Taking $\mu=\overline{\Phi_X}\,\sigma_X^{-1}$ shows
\[
d\E_X=0\quad\Longleftrightarrow\quad \Phi_X=0 .
\]
We write
\[
Z(\Phi)=\{X\in\Teich:\Phi_X=0\}
\]
for the set of critical points of $\E$.

\begin{proposition}
\label{prop:monodromy-energy}
For every $X\in\Teich$,
\[
h_{f^*X}=h_X\circ\phi,\qquad
\E(f^*X)=\E(X),\qquad
\Phi_{f^*X}=\phi^*\Phi_X.
\]
In particular $Z(\Phi)$ is invariant under the infinite cyclic group $\langle f\rangle$.
\end{proposition}

\begin{proof}
The map $\phi:(S,\sigma_{f^*X})\to(S,\sigma_X)$ is an isometry, so $h_X\circ\phi$ is harmonic, and $h_X\circ\phi\simeq j\circ f\simeq j$ by Lemma~\ref{lem:markings}.  Theorem~\ref{thm:ESHartman} gives $h_{f^*X}=h_X\circ\phi$.  The other two identities follow because $\phi$ is an isometry and the Hopf differential is natural under pullback.
\end{proof}

\begin{theorem}[Goldman--Wentworth {\cite[Theorem~C]{GoldmanWentworth}}]
\label{thm:GWproper}
The function $\E$ descends to a proper function
\(\overline{\E}:\Teich/\langle f\rangle\to\R\).
Equivalently, if $X_n\in\Teich$ and $\sup_n\E(X_n)<\infty$, then there are integers $k_n$ such that a subsequence of $(f^{k_n})^*X_n$ converges in $\Teich$.
\end{theorem}

In \cite{GoldmanWentworth} the statement is made for a discrete embedding onto a normal subgroup of a convex cocompact group $\Gamma$, with $\Teich$ divided by the group of mapping classes induced by $\Gamma/\rho(G)$; here that group is $\langle f\rangle$ by Lemma~\ref{lem:markings}.

\begin{theorem}
\label{thm:energy-leastarea}
$\min_{X\in\Teich}\E(X)=A_0$.  Moreover $\E(X)=A_0$ if and only if $h_X$ is a least-area map homotopic to $j$; in that case $h_X$ is a conformal embedding onto a least-area fiber.
\end{theorem}

\begin{proof}
For every smooth $u:(S,\sigma_X)\to M_f$ we have $E_X(u)\geq\Area(u)$, with equality if and only if $u$ is weakly conformal.  Hence
\(\E(X)=E_X(h_X)\geq\Area(h_X)\geq A_0\).
For a least-area fiber and $u_0$, $X_{\min}$ as in Section~4, Corollary~\ref{cor:Hopfzero} gives $h_{X_{\min}}=u_0$, so
\(\E(X_{\min})=E_{X_{\min}}(u_0)=\Area(u_0)=A_0\).
If $\E(X)=A_0$, both inequalities are equalities, so $h_X$ is weakly conformal and has least area in its homotopy class; by the discussion preceding Theorem~\ref{thm:leastarea} it is an embedding onto a least-area fiber.  The converse is the equality $\E(X_{\min})=A_0$.
\end{proof}

\begin{corollary}
\label{cor:minset}
The set of global minima of $\E$ is contained in $Z(\Phi)$.  It is the set of structures $(f^n)^*X_{\min}$, where $\Sigma_{\min}$ ranges over least-area fibers and $n\in\Z$, and its image in $\Teich/\langle f\rangle$ is compact.
\end{corollary}

\begin{proof}
The first two assertions follow from Theorem~\ref{thm:energy-leastarea}, Proposition~\ref{prop:monodromy-energy} and Lemma~\ref{lem:markings}.  The image is the closed subset $\overline{\E}^{-1}(A_0)$ of $\Teich/\langle f\rangle$, which is compact by Theorem~\ref{thm:GWproper}.
\end{proof}

\begin{corollary}
\label{cor:zero-orbit}
For a least-area fiber, the structures $(f^n)^*X_{\min}$, $n\in\Z$, are pairwise distinct points of $Z(\Phi)$, and they leave every compact subset of $\Teich$ as $|n|\to\infty$.  All of them correspond to the same surface $\Sigma_{\min}$.
\end{corollary}

\begin{proof}
A pseudo-Anosov mapping class has no fixed point in $\Teich$ and infinite order, and the mapping class group acts properly discontinuously on $\Teich$.
\end{proof}

\begin{proposition}
\label{prop:minimal-zero-injection}
Let $\Sigma_1,\Sigma_2\subset M_f$ be embedded minimal surfaces isotopic to the fiber, with conformal parametrizations $u_i:(S,X_i)\to\Sigma_i$ homotopic to $j$.  If $X_2=(f^n)^*X_1$ for some $n$, then $\Sigma_1=\Sigma_2$.  Thus $\Sigma\mapsto\langle f\rangle\cdot X_\Sigma$ is an injective map from embedded minimal surfaces isotopic to the fiber to $\langle f\rangle$--orbits in $Z(\Phi)$.
\end{proposition}

\begin{proof}
The maps $u_i$ are conformal and minimal, hence harmonic, so $u_i=h_{X_i}$ by Theorem~\ref{thm:ESHartman}.  By Proposition~\ref{prop:monodromy-energy}, $h_{X_2}=h_{X_1}\circ\phi^n$, so $\Sigma_2=u_2(S)=u_1(S)=\Sigma_1$.  The orbit $\langle f\rangle\cdot X_\Sigma$ does not depend on the choice of $u$ by Lemma~\ref{lem:markings}.
\end{proof}

A point of $Z(\Phi)$ gives a branched minimal immersion homotopic to $j$; without further information it need not be an embedding, so the map of Proposition~\ref{prop:minimal-zero-injection} need not be surjective.

\begin{corollary}
\label{cor:bounded-energy-escape}
Let $X_n\in\Teich$ with $\sup_n\E(X_n)<\infty$.  There are a compact set $K\subset\Teich$ and integers $k_n$ with $X_n\in(f^{-k_n})^*K$ for all $n$.  If $X_n$ leaves every compact subset of $\Teich$, then $|k_n|\to\infty$.
\end{corollary}

\begin{proof}
The first statement is Theorem~\ref{thm:GWproper} together with the proper discontinuity of the action of $\langle f\rangle$.  If $|k_n|$ were bounded along a subsequence, that subsequence would lie in a finite union of translates of $K$, which is compact.
\end{proof}

We use the standard left action of the mapping class group on measured foliations,
\[
f\cdot[\mathcal F]=[(f^{-1})^*\mathcal F].
\]
With this convention, pullback by $\phi^n$ represents the action of $f^{-n}$ on $\mathcal{PMF}(S)$.

\begin{proposition}
\label{prop:Hopf-north-south}
Fix $X\in\Teich$ with $\Phi_X\neq0$, and let $[\mathcal F]\in\mathcal{PMF}(S)$ be the projective class of the vertical measured foliation of $\Phi_X$.  For $X_n=(f^n)^*X$, the vertical foliation of $\Phi_{X_n}=(\phi^n)^*\Phi_X$ has projective class $f^{-n}[\mathcal F]$.  Consequently, as $n\to+\infty$ these classes converge to $[\mathcal F^s]$ unless $[\mathcal F]=[\mathcal F^u]$, and as $n\to-\infty$ they converge to $[\mathcal F^u]$ unless $[\mathcal F]=[\mathcal F^s]$.  The same holds for horizontal foliations.
\end{proposition}

\begin{proof}
The identity for Hopf differentials is Proposition~\ref{prop:monodromy-energy}, and pullback of a quadratic differential by $\phi^n$ pulls back its vertical and horizontal measured foliations.  By the convention just fixed, this is the $f^{-n}$--orbit in $\mathcal{PMF}(S)$.  The asserted limits are therefore the north--south dynamics of the pseudo-Anosov class $f$: $[\mathcal F^u]$ is attracting and $[\mathcal F^s]$ is repelling for positive iterates of $f$ \cite{FLP}.
\end{proof}

By Corollary~\ref{cor:bounded-energy-escape}, sequences of bounded energy leave compact sets only along $\langle f\rangle$--orbits up to a bounded error, and Proposition~\ref{prop:Hopf-north-south} describes the Hopf foliations along these orbits.  Sequences with $\E(X_n)\to\infty$ are not covered by these results.  For large Hopf differentials of harmonic maps into hyperbolic $3$--manifolds see Minsky \cite{Minsky}.

\subsection{Equivariant energy minimization}

Fix $X\in\Teich$ and let $\mathcal A_X$ be the set of smooth $\rho$--equivariant maps $u:\Hh^2\to\Hh^3$.  For $u\in\mathcal A_X$ the function $|du|^2$ is $\rho_X(G)$--invariant, and we set
\[
E_X(u)=\frac12\int_D|du|^2\,dA_{\sigma_X}
\]
for a fundamental domain $D$ of $\rho_X(G)$ with boundary of measure zero; this does not depend on $D$.

\begin{proposition}
\label{prop:variational-CT}
Every $u\in\mathcal A_X$ extends continuously to $\Hh^2\cup S^1$ with boundary map $\partial\iota$.  The functional $E_X$ has a unique minimizer on $\mathcal A_X$, namely $\widetilde h_X$.
\end{proposition}

\begin{proof}
The first statement is Proposition~\ref{prop:anyrepresentative}.  Elements of $\mathcal A_X$ are the lifts of the smooth maps $S\to M_f$ homotopic to $j$ (connect $u$ to $\widetilde j$ by the $\rho$--equivariant geodesic homotopy in $\Hh^3$), and $E_X(u)$ is the energy of the quotient map.  A harmonic map into a manifold of nonpositive curvature minimizes energy in its homotopy class, since energy is convex along geodesic homotopies; every minimizer is harmonic; and Theorem~\ref{thm:ESHartman} gives uniqueness.
\end{proof}

No boundary condition is imposed in Proposition~\ref{prop:variational-CT}: the Peano boundary map is determined by the equivariance.

\section{The Cannon--Thurston relation}

Define an equivalence relation on $\partial G\cong S^1$ by
\[
\xi\sim_{CT}\eta
\quad\Longleftrightarrow\quad
\partial\iota(\xi)=\partial\iota(\eta).
\]
Let $\widetilde\Lambda^s,\widetilde\Lambda^u\subset\Hh^2$ be the lifts of the stable and unstable geodesic laminations of $f$ for a hyperbolic structure $X$.  Two distinct points $\xi,\eta\in S^1$ satisfy $\xi\sim_{CT}\eta$ if and only if they are the endpoints of a leaf of $\widetilde\Lambda^s$ or of $\widetilde\Lambda^u$, or ideal vertices of the same complementary ideal polygon of one of these laminations \cite{CT}; for general degenerate surface groups the corresponding description in terms of ending laminations is due to Mj \cite{MjEnding}.  By Theorems~\ref{thm:harmonicCT} and \ref{thm:minimalPeano}:

\begin{proposition}
\label{prop:sameequiv}
The boundary maps of $\widetilde h_X$ for every $X\in\Teich$, and of the $\rho$--equivariant lift of every least-area fiber, induce the same equivalence relation $\sim_{CT}$ on $\partial G$.
\end{proposition}

The relation $\sim_{CT}$ is topological and carries no metric scale.  The area of a minimal fiber, the $L^2$--norm of its second fundamental form, or its principal curvatures depend on the hyperbolic metric of $M_f$; Propositions~\ref{prop:area-curvature} and \ref{prop:maxcurv} give restrictions valid for every stable minimal fiber.

\begin{question}
Can the hyperbolic metric of $M_f$ together with the stable and unstable measured laminations be used to estimate $A_0$, or the distribution of $|A|^2$ on a least-area fiber, beyond the bounds $2\pi(g-1)<A_0<4\pi(g-1)$ and $\min|A|^2<2\le\max|A|^2$?
\end{question}

\begin{question}
Let $X_n\in\Teich$ with $\E(X_n)\to\infty$.  Which projective measured foliations occur as limits of the classes of the vertical foliations of $\Phi_{X_n}$?  Under what hypotheses are the only limits $[\mathcal F^s]$ and $[\mathcal F^u]$, and when do the rescaled equivariant maps $\widetilde h_{X_n}$ converge to an equivariant map to the dual $\R$--tree?
\end{question}

\section{Relation with escaping harmonic maps}

Pereira do Vale and Verjovsky \cite{PV} constructed harmonic maps whose lifts have space-filling boundary behavior.  In the fibered hyperbolic setting, Theorem~\ref{thm:harmonicCT} identifies the boundary map with the Cannon--Thurston map, and Proposition~\ref{prop:anyrepresentative} shows that it depends only on the fiber subgroup inclusion: equivariant lifts of homotopic maps are a bounded distance apart.  The harmonic map selects, for each conformal structure on the fiber, a representative of the homotopy class; a least-area fiber is a representative that determines its own conformal structure $X_{\min}$, which is a global minimum of $\E$.

\section*{Funding}
This work was supported by Proyecto PAPIIT, Direcci\'on General de Asuntos del Personal Acad\'emico, Universidad Nacional Aut\'onoma de M\'exico [grant number IN103324].

\section*{Acknowledgments}
I used ChatGPT (OpenAI) and Claude (Anthropic) for proofreading, editorial revision, and auxiliary checks of calculations.  I independently verified the mathematical content and am responsible for the final manuscript.

\bigskip
\noindent
\textsc{Alberto Verjovsky}\\
Instituto de Matem\'aticas, Universidad Nacional Aut\'onoma de M\'exico\\
Unidad Cuernavaca, M\'exico\\
\texttt{albertoverjovsky@gmail.com}

\end{document}